\documentclass[11pt]{amsart}
\usepackage{amsmath}
\usepackage{amssymb} 
\usepackage{pgf}
\usepackage{tikz}
\usetikzlibrary{matrix}
\usetikzlibrary{arrows}
\tikzstyle{place}=[draw,circle,minimum size=1.2mm,inner sep=1pt,outer sep=-1.1pt,fill=black]
\tikzset{>=latex,shorten >= 0.01cm,shorten <= 0.01cm}

\newcommand{\oast}{\circledast}	
\DeclareMathOperator{\sgn}{sgn}

\newcommand{\A}{\mathcal{A}}

\newcommand{\cA}{\mathcal{A}}\newcommand{\cT}{\mathcal{T}}

\newcommand{\cB}{\mathcal{B}}
\newcommand{\cV}{\mathcal{V}}

\newcommand{\cC}{\mathcal{C}}

\newcommand{\cW}{\mathcal{W}}

\newcommand{\ri}[1]{\mathop{\rm ri}(#1)} 

\newtheorem{theorem}{Theorem}[section] 
 
\newtheorem{corollary}[theorem]{Corollary}

\newtheorem{example}[theorem]{Example}

\newtheorem{remark}[theorem]{Remark}

\begin{document}
\title{Spectrally Arbitrary Pattern Extensions}
\thanks{Last updated: December 7, 2016}

\author{In-Jae Kim} 
\address{Department of Mathematics and Statistics, Minnesota State University, Mankato, MN, USA, 56001} 
\email{in-jae.kim@mnsu.edu}

\author{Bryan~L. Shader}
\address{Department of Mathematics, University of Wyoming, Laramie, WY, USA, 82071-3036} 
\email{bshader@uwyo.edu.}

\author{Kevin N. Vander~Meulen}
\address{Department of Mathematics, Redeemer University College, Ancaster, ON, Canada, L9K 1J4}
\email{kvanderm@redeemer.ca}

\author{Matthew West}
\address{Department of Mathematics, Redeemer University College, Ancaster, ON, Canada, L9K 1J4. Current address: TMV Control Systems Inc., Cambridge, ON, Canada, N1R 5R4}
\email{matthewmcwest@gmail.com}

\keywords{spectrally arbitrary patterns, inertially arbitrary patterns}  
\subjclass[2010]{15A18, 15A29, 15B35}

\begin{abstract}
A matrix pattern is often either a sign pattern with entries in $\{0, +,-\}$ or, more simply, a nonzero pattern
with entries in $\{ 0,* \}$. 
A matrix pattern $\cA$ is spectrally arbitrary if for any choice of a real matrix
spectrum, there is a real matrix having the pattern $\cA$ and the chosen spectrum. 
We describe a graphical technique, a triangle extension, for constructing spectrally arbitrary patterns out of some known 
lower order spectrally arbitrary patterns. These methods provide a new way of viewing some known spectrally
arbitrary patterns, as well as providing many new families of spectrally arbitrary patterns. We also
demonstrate how the technique can be applied to certain inertially arbitrary patterns
to obtain larger inertially arbitrary patterns. We then provide an additional extension method 
for zero-nonzero patterns, patterns with entries in $\{ 0, *, \oast\}$.
\end{abstract}
\maketitle 


\tikzstyle{place}=[circle,draw=black!100,fill=black!100,thick,inner sep=0pt,minimum size=1mm]
\tikzstyle{left}=[>=latex,<-,semithick]
\tikzstyle{right}=[>=latex,->,semithick]
\tikzstyle{nleft}=[>=latex,-,semithick]
\tikzstyle{nright}=[>=latex,-,semithick]
\tikzstyle{right2}=[-,semithick]

\section{Introduction}
 
The idea of exploring spectrally arbitrary patterns originated in \cite{Drew} in the context of  
inverse eigenvalue and eigenvalue completion problems. Since then,
many different papers that have described classes of spectrally arbitrary patterns
(for a survey of some of these, see~\cite{Catral}). 
Some papers have used a digraph to help describe the structure of the pattern (e.g. the star
digraphs~\cite{MTvdD}). Also,
some papers have used a digraph to observe necessary conditions on a spectrally
arbitrary pattern, such as the need to have a two-cycle in the digraph \cite{CV}. 
A nilpotent-centralizer method was recently introduced in \cite{GS} and used to demonstrate that
a certain tridiagonal pattern, having a digraph whose underlying graph is a path,
is spectrally arbitrary. This nilpotent-centralizer method is equivalent to the well-used nilpotent-Jacobian
method introduced in~\cite{Drew}; but this centralizer method avoids direct interaction with a Jacobian matrix.   

Here we introduce a digraph method for building spectrally arbitrary patterns from smaller
patterns, assuming the smaller patterns were demonstrated to be spectrally arbitrary via either the nilpotent-Jacobian method or  
the nilpotent-centralizer method. We start by describing formal definitions and the nilpotent-Jacobian method in Section~\ref{def}.
In Section~\ref{triangleSAP} we introduce a new method, called a triangle extension, that 
avoids the need to calculate a new Jacobian or centralizer.
Using this method, we can reconstruct some known families of spectrally arbitrary patterns,
as well as build many new families. Using recent work of \cite{CF} and \cite{CGKOVV}, we also apply these methods to inertially arbitrary patterns in Section~\ref{triangleIAP}. 
We conclude with an additional extension method for $\{0,*, \oast\}$-patterns in Section~\ref{znz}.

\section{Background definitions and the nilopotent-Jacobian method}\label{def}

A \emph{pattern} is any matrix $\A=[\A_{ij}]$ of order $n$ with entries in $\{0,+,-,*,\oast\}$.
In this paper we discuss nonzero patterns, sign patterns, and zero-nonzero patterns. A \emph{nonzero
pattern} $\A$ is a pattern with entries in $\{ 0, * \}$. 
A real matrix $A=[a_{ij}]$ has nonzero pattern $\A$, and we say $A\in \A$, if $a_{ij}=0$ whenever $\A_{ij}=0$ and $a_{ij}$ is nonzero whenever $\A_{ij}=*$.
A \emph{sign pattern} $\A$ is a pattern with entries in $\{ 0, +, - \}$. 
A real matrix $A=[a_{ij}]$ has sign pattern $\A$, and we say $A\in \A$, 
if $a_{ij}=0$ whenever $\A_{ij}=0$,  $a_{ij}>0$ if $\A_{ij}=+$, and
$a_{ij}<0$ if $\A_{ij} =-$. 
In the last section we extend our results to \emph{zero-nonzero patterns}, matrices
with entries in $\{0,*, \oast\}$. A real matrix $A=[a_{ij}]$ has zero-nonzero pattern $\A$, 
and we say $A\in \A$, if $a_{ij}=0$ whenever $\A_{ij}=0$ and $a_{ij}$ is nonzero 
whenever $\A_{ij}=*$. Note that, if $\A_{ij}=\oast$ then there is no restriction on the real number $a_{ij}$
for a matrix $A\in \A$.

The digraph $D(\A)$ of a pattern has vertices $\{ 1,2,\ldots,n\}$ with an arc $(i,j)$ from
vertex $i$ to vertex $j$ precisely when $\A_{ij}\neq 0$. The \emph{graph} of a pattern $\cA$ is the underlying simple graph
of $D(\A)$.
A (simple) \emph{cycle} in a pattern $\A$ (or in a matrix $A$ with pattern $\A$) corresponds to a cycle in the digraph $D(\A)$. 
For $k\geq 1$, we say a  \emph{$($simple$)$ $k$-cycle} of $A$ is a nonzero product 
$a_{i_1,i_2}a_{i_2,i_3}\cdots a_{i_{k-1},i_k}a_{i_k,i_1}$ 
with $k$ distinct indices $\{i_1,i_2,\ldots,i_k\}$. We call a $1$-cycle a \emph{loop}.
A \emph{composite $k$-cycle} is a product of disjoint cycles of $A$ using 
nonzero entries from exactly $k$ rows (and corresponding columns) of $A$.

The \emph{sign} of a simple $k$-cycle is $(-1)^{k-1}$. The sign of a composite cycle is the 
product of the signs
of its simple cycles. The next observation (see e.g.~\cite[Section 9.1]{BR}) has been useful in analyzing the 
implications of the combinatorial 
structure of a matrix $A$. In particular, 
the characteristic polynomial of $A$ is of the form
\begin{equation}\label{charpoly}
p_A(z)=z^n-E_1z^{n-1}+E_2z^{n-2}-\cdots +(-1)^nE_n
\end{equation}
where $E_k$ is the sum of all signed $k$-cycles (simple and composite).

Let $\A$ be an $n \times n$ sign pattern, and suppose that
there exists a nilpotent matrix $A \in \A$ with 
$m\geq n$ nonzero entries, say $a_{i_1,j_1},\ldots,a_{i_m,j_m}$.  
Let
$X=X(x_1,\ldots,x_m)$ denote the matrix obtained from $A$ with the $(i_k,j_k)$-th position
replaced with the variable $x_k$ for $k=1,\ldots,m$, and let
\[
{p}_{X}(z) = z^n + f_1z^{n-1} + f_2z^{n-2}+
\cdots + f_{n-1}z + f_n\]
be the characteristic polynomial of $X(x_1,\ldots,x_m)$, where each
$f_i = f_i(x_1,\ldots,x_m)$ is a polynomial in $x_1,\ldots,x_m$.
Let $J=J_X$ be the Jacobian matrix with
$J_{i,j}=\frac{\partial f_i}{\partial x_j},$
for $1\leq i \leq n$, and $1\leq j \leq m$.
Setting  $\mathbf{x} = (x_1,\ldots,x_m)$ and $\mathbf{a}=(a_{i_1,j_1},\ldots,a_{i_m,j_m})$,
we let 
$J_{X=A}=J|_{\mathbf{x}=\mathbf{a}}$
 denote the Jacobian matrix evaluated at 
${\mathbf{x}=\mathbf{a}}$.
The \emph{{nilpotent-Jacobian method}} is to seek a nilpotent matrix $A\in \A$ for which 
$J_{X=A}$ has rank $n$. In this case, we say that $A$ \emph{allows a full-rank Jacobian}. 
The following theorem, first developed in~\cite{Drew}, shows
that $\A$ is spectrally arbitrary if such a realization is found.  

Recall that $\cB$ is a \emph{superpattern} of a pattern $\A$ if $\A_{i,j}\neq 0$ implies $\cB_{i,j}=\A_{i,j}$.
Note that $\A$ is a superpattern of itself.

\begin{theorem}[\cite{Drew}]\label{NJ}
If a nilpotent matrix $A\in \cA$ allows a full-rank Jacobian, 
then every superpattern of $\mathcal{A}$ is spectrally arbitrary.
\end{theorem}

\section{Triangle extensions of spectrally arbitrary sign patterns}\label{triangleSAP}

If $G$ is a digraph having an arc $(u,v)$, $u\neq v$,   
such that there is a loop at $u$ 
then a \emph{{triangle extension}} of $G$ at $(u,v)$ is the graph $H$ obtained from $G$
by inserting a new vertex $w$, moving the loop to $w$, 
and inserting arcs $(u,w)$ 
and $(w,v)$ as in Figure~\ref{tri}. \begin{figure}
\begin{center}
\begin{tabular}{cc}
\begin{tikzpicture}
    \node (1) at (1,0)[place]{};
    \node (1b) at (1,-0.3){$u$}; 
    \node (2) at (3,0)[place]{}; 
    \node (2b) at (3,-0.3){$v$};
\draw [->] (1) to  (2);
\path[every node] (1)       edge[loop]node{}();
\end{tikzpicture}
\qquad & \qquad
\begin{tikzpicture}
    \node (1) at (1,0)[place]{}; 
    \node (1b) at (1,-0.3){$u$};
    \node (3) at (2.1,1)[place]{}; 
    \node (2) at (3.2,0)[place]{}; 
    \node (2b) at (3.2,-0.3){$v$};
    \node (3b) at (2.4, 1){$w$};
\draw [->] (1) to (2);
\draw[->] (1) to  (3);
\draw[->] (3) to  (2);
\path[every node] (3) edge[loop]node{}();
\end{tikzpicture}
\end{tabular}
\end{center}
\caption{A triangle extension along arc $(u,v)$.}\label{tri}
\end{figure}
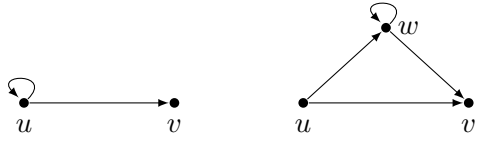
If $\cA$ and $\cB$ are patterns corresponding
to the digraphs $G$ and $H$ respectively, we will say $\cB$ is a \emph{triangle extension}
of $\cA$ at $(u,v)$. If $\cB$ is a sign-pattern, then $\cB$ is a \emph{signed} triangle extension 
of $\cA$ if, in addition, $\sgn(\cB_{u,w})\sgn(\cB_{w,v})=\sgn(\cA_{u,u})\sgn(\cA_{u,v})$.
Observe that any composite cycle of $G$ will also be a composite cycle of $H$ (replacing the loop at $u$ by the loop at $w$ as necessary).  But $H$ can have some additional cycles.

\begin{theorem}\label{triangle}
Suppose $\cA$ is an order $n$ sign pattern and $A\in\A$ is a nilpotent matrix that allows a full-rank Jacobian. 
Let $(u,v)$, $u\neq v$,  be an arc on a composite $n$-cycle of $D(\cA)$ with a loop at $u$.
Suppose  
\begin{enumerate}
\item[(i)] that each cycle of $D(\cA)$ that contains $u$ 
also contains $(u,u)$ or $(u,v)$, and 
\item[(ii)] the sum of the weights of the composite $n$-cycles using edge $(u,v)$ in $A$ is nonzero.
\end{enumerate}
If  $\cB$ is a signed triangle extension of $\cA$ at $(u,v)$, 
then every superpattern of $\cB$ is spectrally arbitrary. 
\end{theorem}

\begin{proof} 
Labeling the vertices of $\A$ as $1,2,\ldots, n$, assume $u$ and $v$ are vertices $n$ and $n-1$ respectively.

Let $A\in\cA$ be a nilpotent matrix and $X$ be a matrix having variables 
in $m$ positions, labeled $x_1,x_2,\ldots,x_m$ 
such that the Jacobian $J_{X=A}$ has rank $n$. 
For simplicity, we let $k=X_{n,n}$ and $d=X_{n,n-1}$.

Let $\cB$ be a signed-triangle extension of $\cA$. Without loss of generality, assume $\cB_{n+1,n}=+$.
Let $B\in\cB$ with $B_{ij}=A_{ij}$ for $1\leq i,j\leq n$, except $B_{n,n}=0$, 
 $B_{n+1,n+1}=A_{n,n}$, $B_{n+1,n-1}=1$ and $B_{n,n+1}=A_{n,n-1}A_{n,n}$. 
Let $Y$ be obtained from $B$ by placing the variables $x_1,x_2,\ldots,x_m$ in the 
same positions as they appear in $X$ (except for $X_{n,n}$) with $Y_{n+1,n+1}=k$. Further let $Y_{n,n+1}=x_{m+1}$. 
Since any composite cycle of $D(\cA)$ will appear as a composite cycle of $D(\cB)$, 
it follows from~(\ref{charpoly}) that many summands of the characteristic polynomial $p_Y$ can be obtained from
$p_X$. Using condition~(i), we know the remaining summands can be paired based on composite cycles
that use arc $(u,v)$ and those that use $(w,w)$. Thus
\begin{equation}\label{charY}
{p}_{Y}(z) = zp_X(z) + 
\sum_{r=3}^{n}S_r(x_{m+1}-dk)z^{n-r+1} +S_{n+1}(x_{m+1}-dk)
\end{equation}
for some polynomials $S_r$, $3\leq r\leq n+1$, of the variables $x_1,\ldots,x_m$ (not including $d$ and $k$). 
Using (\ref{charY}), we can see that $B$ is nilpotent since $A$ is nilpotent and $x_{m+1}=dk$ in $B$.

In considering the Jacobian matrix $J_Y$ , we may assume that
$Y_{n,n-1}=x_{m-1}$ and $Y_{n+1,n+1}=x_{m}$. 
In this case, the Jacobian  matrix $J_Y$ is of the form:
$$J_Y=\left[\begin{array}{cc}
J_X & \mathbf{0}\\
\mathbf{0}^T&0
\end{array}\right]
+(x_{m+1}-dk)M+ \mathbf{s} \mathbf{t}^T
$$
for some matrix $M$, $\mathbf{s}^T=[0,0,S_3,\ldots,S_{n+1}]$ and $\mathbf{t}^T=[0,\ldots,0, -k,-d,1]$.
Note that when considering $J_{Y=B}$, we can ignore $M$ since $x_{m+1}=dk$ in $B$. Further, 
by condition~(ii), we know that $S_{n+1}\neq 0$.  
It now follows that $J_{Y=B}$ is column equivalent to a block  triangular matrix with block $J_{X=A}$ and block $[1]$.
Thus, by Theorem~\ref{NJ}, every superpattern of $\cB$ is spectrally arbitrary.
\end{proof}

\begin{remark}\label{det}{\rm 
The condition (i) of Theorem~\ref{triangle} is equivalent to requiring 
that $u$ is the only vertex in
the strong component of $D(A)-(u,v)$ that contains $u$ (for the definition of strong component, see e.g.~\cite[p.54]{BR}).
Further, if we let $A(u,v)$ denote the matrix obtained from $A$ by deleting row $u$ and column $v$, then 
condition (ii) of Theorem~\ref{triangle} is equivalent to saying that 
$\det{A(u,v)} \neq 0$.
}\end{remark}

\begin{remark}{\rm 
While a triangle extension is defined for a loop at the tail of $(u,v)$, it could just as well have been
defined using a loop at the head: this can be observed by noting such is an equivalent pattern via transposition. Thus a loop
could be moved from either $u$ or $v$ when creating a triangle extension.} 
\end{remark}

\begin{remark}{\rm
Focusing on the loop at $w$, it can be observed that the new arcs $(u,w)$ and $(w,v)$ 
constructed via triangle extension using Theorem~\ref{triangle} will automatically
satisfy the conditions of the theorem if applied to the new pattern. Thus, starting with an irreducible spectrally arbitrary 
pattern $\cA$ of order $n$ and using Theorem~\ref{triangle}, one could recursively construct various patterns of any 
order $m\geq n$. 
}
\end{remark}

\begin{example} {\rm
For $n\geq 4$, 
the lower Hessenberg pattern $\cV_n$ in~\cite{Britz} can be seen to be a signed spectrally arbitrary triangle extension of $\cV_{n-1}$.
(In fact, $\cV_3$ is a triangle extension
of $\cT_2$ described below.)}
\end{example}

In~\cite{D}, it was shown that a certain tridiagonal pattern $\cT_n$, whose graph is a path with loops on the ends, 
is potentially nilpotent. Then~\cite{GS} demonstrated that
$\cT_n$ is spectrally arbitrary using a new technique equivalent to the nilpotent-Jacobian method.
There are a variety of ways of recursively constructing classes of spectrally arbitrary patterns 
via triangle extension starting with $\cT_n$. 
To simplify our next examples, we will ignore signs and describe some nonzero patterns obtained recursively
by triangle extension.

\begin{example}{\rm 
Starting with $\cT_2$, one can recursively construct a class $\cT_{2,n}$, $n\geq 1$,  of spectrally arbitrary patterns of order $n+2$ whose graph looks like a triangulated beam. Several examples are drawn in Figure~\ref{C3}. Note that these provide pentadiagonal spectrally arbitrary patterns if, for example, we label the top vertices consecutively with odd integers and the bottom vertices with even integers. 
\begin{figure}[ht]
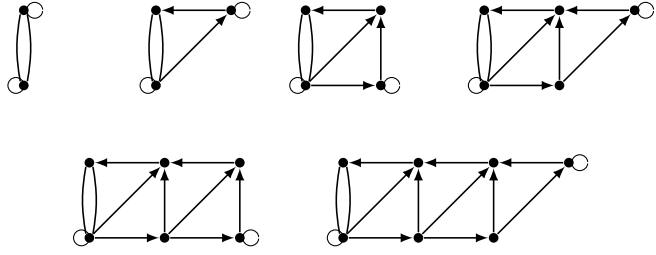

\begin{center}
\tikzpicture
\node (1) at (-0.5,0.5)[place] {};
\node (4) at (-0.5,-0.5)[place] {};
\draw [-] (-0.5,-0.5) arc (360:0:3pt);
\draw [-] (-0.25, 0.5) arc (360:0:3pt);
\draw [nleft] (1.south west) to [bend right=10] (4.north west);
\draw [nleft] (4.north east)to [bend right=10] (1.south east);
\endtikzpicture
\hspace{3em}
\tikzpicture
\node (1) at (-0.5,0.5)[place] {};
\node (2) at (0.5,0.5)[place] {};
\node (4) at (-0.5,-0.5)[place] {};
\draw [right] (2) to (1);
\draw [right] (4) to (2);
\draw [-] (-0.5,-0.5) arc (360:0:3pt);
\draw [-] (0.75,0.5) arc (360:0:3pt);
\draw [nleft] (1.south west) to [bend right=10] (4.north west);
\draw [nleft] (4.north east)to [bend right=10] (1.south east);
\endtikzpicture
\hspace{1em}
\tikzpicture
\node (1) at (-0.5,0.5)[place] {};
\node (2) at (0.5,0.5)[place] {};
\node (3) at (0.5,-0.5)[place] {};
\node (4) at (-0.5,-0.5)[place] {};
\draw [right] (2) to (1);
\draw [right] (3) to (2);
\draw [right] (4) to (2);
\draw [right] (4) to (3);
\draw [-] (-0.5,-0.5) arc (360:0:3pt);
\draw [-] (0.75,-0.5) arc (360:0:3pt);
\draw [nleft] (1.south west) to [bend right=10] (4.north west);
\draw [nleft] (4.north east)to [bend right=10] (1.south east);
\endtikzpicture
\hspace{2em}
\tikzpicture
\node (1) at (-0.5,0.5)[place] {};
\node (2) at (0.5,0.5)[place] {};
\node (3) at (0.5,-0.5)[place] {};
\node (4) at (-0.5,-0.5)[place] {};
\node (6) at (1.5,0.5)[place] {};
\draw [right] (3) to (6);
\draw [right] (6) to (2);
\draw [-] (1.75,0.5) arc (360:0:3pt);
\draw [right] (2) to (1);
\draw [right] (3) to (2);
\draw [right] (4) to (2);
\draw [right] (4) to (3);
\draw [-] (-0.5,-0.5) arc (360:0:3pt);
\draw [nleft] (1.south west) to [bend right=10] (4.north west);
\draw [nleft] (4.north east)to [bend right=10] (1.south east);
\endtikzpicture

\vspace{2em}
\tikzpicture
\node (1) at (-0.5,0.5)[place] {};
\node (2) at (0.5,0.5)[place] {};
\node (3) at (0.5,-0.5)[place] {};
\node (4) at (-0.5,-0.5)[place] {};
\node (5) at (1.5,-0.5)[place] {};
\node (6) at (1.5,0.5)[place] {};
\draw [right] (3) to (5);
\draw [right] (3) to (6);
\draw [right] (6) to (2);
\draw [right] (5) to (6);
\draw [right] (2) to (1);
\draw [right] (3) to (2);
\draw [right] (4) to (2);
\draw [right] (4) to (3);
\draw [-] (-0.5,-0.5) arc (360:0:3pt);
\draw [nleft] (1.south west) to [bend right=10] (4.north west);
\draw [nleft] (4.north east)to [bend right=10] (1.south east);
\draw [-] (1.75,-0.5) arc (360:0:3pt);
\endtikzpicture
\hspace{2em}
\tikzpicture
\node (1) at (-0.5,0.5)[place] {};
\node (2) at (0.5,0.5)[place] {};
\node (3) at (0.5,-0.5)[place] {};
\node (4) at (-0.5,-0.5)[place] {};
\node (5) at (1.5,-0.5)[place] {};
\node (6) at (1.5,0.5)[place] {};
\node (7) at (2.5,0.5)[place] {};
\draw [right] (3) to (5);
\draw [right] (3) to (6);
\draw [right] (6) to (2);
\draw [right] (5) to (6);
\draw [right] (5) to (7);
\draw [right] (7) to (6);
\draw [right] (2) to (1);
\draw [right] (3) to (2);
\draw [right] (4) to (2);
\draw [right] (4) to (3);
\draw [-] (-0.5,-0.5) arc (360:0:3pt);
\draw [nleft] (1.south west) to [bend right=10] (4.north west);
\draw [nleft] (4.north east)to [bend right=10] (1.south east);
\draw [-] (2.75,0.5) arc (360:0:3pt);
\endtikzpicture
\end{center} 
\caption{$\cT_2$ and consecutive spectrally arbitrary triangle extensions $\cT_{2,1}$, $\cT_{2,2}$, $\cT_{2,3}$, $\cT_{2,4}$ and $\cT_{2,5}$.}\label{C3}
\end{figure}
}\end{example}

\begin{example}{\rm 
Starting with $\cT_n$, $n\geq 2$, one can recursively construct a class $\cT_{n,m}$, $m\geq 1$ of spectrally arbitrary patterns of order $n+m$ whose graph looks like a triangulated beam extension of $\cT_n$: three examples appear in Figure~\ref{BNM}. 
\begin{figure}[ht]
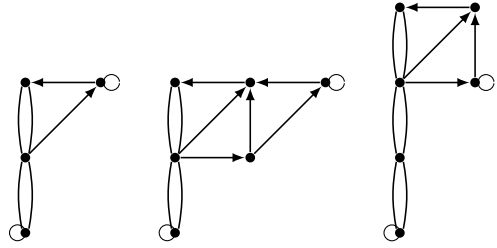

\begin{center}
\tikzpicture
\node (1) at (-0.5,0.5)[place] {};
\node (2) at (0.5,0.5)[place] {};
\node (4) at (-0.5,-0.5)[place] {};
\node (8) at (-0.5,-1.5)[place]{};
\draw [right] (2) to (1);
\draw [right] (4) to (2);
\draw [-] (-0.5,-1.5) arc (360:0:3pt);
\draw [-] (0.75,0.5) arc (360:0:3pt);
\draw [nleft] (1.south west) to [bend right=10] (4.north west);
\draw [nleft] (4.north east)to [bend right=10] (1.south east);
\draw [nleft] (4.south west) to [bend right=10] (8.north west);
\draw [nleft] (8.north east)to [bend right=10] (4.south east);
\endtikzpicture
\hspace{1em}
\tikzpicture
\node (1) at (-0.5,0.5)[place] {};
\node (2) at (0.5,0.5)[place] {};
\node (3) at (0.5,-0.5)[place] {};
\node (4) at (-0.5,-0.5)[place] {};
\node (6) at (1.5,0.5)[place] {};
\node (8) at (-0.5,-1.5)[place] {};
\draw [right] (3) to (6);
\draw [right] (6) to (2);
\draw [-] (1.75,0.5) arc (360:0:3pt);
\draw [right] (2) to (1);
\draw [right] (3) to (2);
\draw [right] (4) to (2);
\draw [right] (4) to (3);
\draw [-] (-0.5,-1.5) arc (360:0:3pt);
\draw [nleft] (1.south west) to [bend right=10] (4.north west);
\draw [nleft] (4.north east)to [bend right=10] (1.south east);
\draw [nleft] (4.south west) to [bend right=10] (8.north west);
\draw [nleft] (8.north east)to [bend right=10] (4.south east);
\endtikzpicture
\hspace{1em}
\tikzpicture
\node (1) at (-0.5,0.5)[place] {};
\node (2) at (0.5,0.5)[place] {};
\node (3) at (0.5,-0.5)[place] {};
\node (4) at (-0.5,-0.5)[place] {};
\node (8) at (-0.5,-1.5)[place] {};
\node (9) at (-0.5,-2.5)[place] {};
\draw [right] (2) to (1);
\draw [right] (3) to (2);
\draw [right] (4) to (2);
\draw [right] (4) to (3);
\draw [-] (-0.5,-2.5) arc (360:0:3pt);
\draw [-] (0.75,-0.5) arc (360:0:3pt);
\draw [nleft] (1.south west) to [bend right=10] (4.north west);
\draw [nleft] (4.north east)to [bend right=10] (1.south east);
\draw [nleft] (4.south west) to [bend right=10] (8.north west);
\draw [nleft] (8.north east)to [bend right=10] (4.south east);
\draw [nleft] (8.south west) to [bend right=10] (9.north west);
\draw [nleft] (9.north east)to [bend right=10] (8.south east);
\endtikzpicture
\hspace{2em}

\end{center}
\caption{Spectrally arbitrary triangle extensions $\cT_{3,1},\cT_{3,3}$ and $\cT_{4,2}$.}\label{BNM}
\end{figure}
}\end{example}

\begin{remark}\label{hyp} {\rm If the hypotheses (i) and (ii) of Theorem~\ref{triangle} do not hold, then a triangle extension
of a spectrally arbitrary pattern could fail to give a spectrally arbitrary pattern. For example
the triangle extension $\cV'$ in Figure~\ref{V} is not a spectrally arbitrary nonzero pattern 
(see~\cite[Prop. 2.4]{CV2}). 
While the two hypotheses in Theorem~\ref{triangle} are sufficient conditions, they are not necessary conditions to 
obtain a spectrally arbitrary pattern: the
pattern $\cV''$ in Figure~\ref{V} is spectrally arbitrary (see the second pattern in the last row of Appendix A of \cite{CM})
and is a triangle extension of $\cV_3$ which fails condition (i). 
\begin{figure}[ht]
\begin{center}
\tikzpicture
\node (1) at (-0.5,0.5)[place] {};
\node (2) at (0.5,0.5)[place] {};
\node (4) at (-0.5,-0.5)[place] {};

\draw [right] (2) to (1);
\draw [right] (4) to (2);

\draw [-] (-0.5,0.5) arc (360:0:3pt);
\draw [-] (0.75,0.5) arc (360:0:3pt);
\draw [nleft] (1.south west) to [bend right=10] (4.north west);
\draw [nleft] (4.north east)to [bend right=10] (1.south east);
\endtikzpicture
\hspace{2em}
\tikzpicture
\node (1) at (-0.5,0.5)[place] {};
\node (2) at (0.5,0.5)[place] {};
\node (4) at (-0.5,-0.5)[place] {};
\node (3) at (-1.5,0.5)[place]{};

\draw [right] (2) to (1);
\draw [right] (4) to (2);
\draw [right] (3) to (4);
\draw [right] (1) to (3);

\draw [-] (-1.5,0.5) arc (360:0:3pt);
\draw [-] (0.75,0.5) arc (360:0:3pt);
\draw [nleft] (1.south west) to [bend right=10] (4.north west);
\draw [nleft] (4.north east)to [bend right=10] (1.south east);
\endtikzpicture
\hspace{2em}
\tikzpicture
\node (1) at (-0.5,0.5)[place] {};
\node (2) at (0.5,0.5)[place] {};
\node (4) at (-0.5,-0.5)[place] {};
\node (3) at (-1.5,0.5)[place]{};

\draw [right] (2) to (1);
\draw [right] (4) to (2);
\draw [right] (4) to (3);
\draw [right] (3) to (1);

\draw [-] (-1.5,0.5) arc (360:0:3pt);
\draw [-] (0.75,0.5) arc (360:0:3pt);
\draw [nleft] (1.south west) to [bend right=10] (4.north west);
\draw [nleft] (4.north east)to [bend right=10] (1.south east);
\endtikzpicture
\caption{$\cV_3$ and triangle extensions $\cV'$ and $\cV''$.}\label{V}
\end{center}
\end{figure}}
\end{remark}
\section{Triangle extensions for inertially arbitrary patterns}\label{triangleIAP}

The \emph{inertia} of a matrix $A$ is the ordered triple $(a,b,c)$ where $a$ (resp. b) is the number of eigenvalues
of $A$ with positive (resp. negative) real part, and $c$ is the number of eigenvalues with zero real part.
The \emph{refined inertia} of $A$ is $\ri{A}=(a,b,c_1,c_2)$ with $a$ and $b$ as above, but $c_1$ is the number of eigenvalues
of $A$ that are zero and $c_2$ is the number of eigenvalues that are purely imaginary (and hence $c=c_1+c_2$).
An order $n$ pattern $\cA$ is \emph{inertially arbitrary} if for every choice of $(a,b,c)$ with $a+b+c=n$ there is some matrix 
$A\in \cA$ with inertia $(a,b,c)$.
We next make use of the following theorem to apply triangle extensions to 
inertially arbitrary patterns. 

\begin{theorem}{\rm\cite[Theorem $2.13$]{CF}}\label{NJ-IAP}
Let $A$ be an order $n$ matrix with pattern $\cA$ and $\ri{A}=(0,0,c_1,c_2)$ for some $c_1\geq 2$.
If $A$ allows a full-rank Jacobian
then every superpattern of $\cA$ is inertially arbitrary.
\end{theorem}

\begin{theorem}\label{triangle2}
Given pattern $\cA$, suppose $A\in \A$ is a matrix with refined inertia $\ri{A}=(0,0,c_1,c_2)$ for some $c_1\geq 2$
and suppose $A$ allows a full-rank Jacobian.
Let $(u,v)$, $u\neq v$, be an arc on a composite $n$-cycle of $D(\cA)$ with a loop at $u$.
Suppose  
\begin{enumerate}
\item[(i)] that any cycle that contains $u$ 
also contains $(u,u)$ or $(u,v)$, and 
\item[(ii)] the sum of the weights of the composite $n$-cycles using edge $(u,v)$ in $A$ is nonzero.
\end{enumerate}
If  
$\cB$ is a (signed) triangle extension of $\cA$ at $(u,v)$, 
then  
every superpattern of $\cB$ is inertially arbitrary. 
\end{theorem}

\begin{proof}
The proof of the result is the same as Theorem~\ref{triangle} except we need to observe that
 the resultant triangle extension will
have refined inertia $(0,0,c_1+1,c_2)$,
and that we need to apply Theorem~\ref{NJ-IAP} instead of Theorem~\ref{NJ}.
\end{proof}

\begin{remark}{\rm 
As with Theorem~\ref{triangle}, Theorem~\ref{triangle2} can be applied recursively to the
new edges of a triangle extension.
}
\end{remark}

\begin{example} {\rm 
For this example, we recall certain inertially arbitrary patterns
$\cW^*_n$ and $\cW_n$ introduced in \cite{CVV} and discussed in \cite{CF} and \cite{CGKOVV}.
We observe here that Theorem~\ref{triangle2} reduces the work
needed to determine that these patterns are inertially arbitrary.
In particular, note that the pattern $\cW^*_{n}$ (and $\cW_{n+1}$)
 is a triangle extension of $\cW^*_{n-1}$ (resp. signed triangle extension of $\cW_{n}$)
 for all $n\geq 6$. Further, $\cW^*_5$ 
 allows refined inertia $(0,0,1,4)$ as noted in \cite{CGKOVV} 
 (and likewise $\cW_6$  
 allows refined inertia $(0,0,2,4)$ as noted in \cite{CF}). By recursive application
 of Theorem~\ref{triangle2}, we can observe that every superpattern
 of  $\cW_n$ and $\cW^*_{n-1}$ is inertially arbitrary for all $n\geq 6$. 
}\end{example}

\section{Concluding remarks: zero-nonzero pattern extensions}\label{znz}

We now develop a graph extension result for zero-nonzero patterns.
Recall that a zero-nonzero pattern $\A$ has entries in $\{0, *, \oast \}$.
A real matrix $A$ has zero-nonzero pattern $\cA$ if $A_{i,j}=0$ implies $\cA \neq *$ and 
$A_{i,j}\neq 0$ implies $\cA\neq 0$.
The triangle extension method of the previous sections works for zero-nonzero patterns. But there
is more flexibility that can be achieved for extending zero-nonzero patterns.
One adjustment is that we do not move a loop in the zero-nonzero pattern extension
described in Theorem~\ref{triangle3}.  In particular, given $u,v$ are distinct
vertices of $G$, we say that digraph $H$ is a \emph{$\oast$-extension} of $G$ at $(u,v)$
 if $H$ can be obtained from $G$ by adding a vertex $w$ and the two arcs $(u,w)$ and $(w,v)$ with at least one of the new arcs
labeled $\oast$ with the other label in $\{\oast,*\}$. 
If $\cA$ and $\cB$ are patterns corresponding
to the digraphs $G$ and $H$ respectively, we will say $\cB$ is a \emph{$\oast$-extension}
of $\cA$ at $(u,v)$.
We use the notation of Remark~\ref{det} in the following theorem.

\begin{theorem}\label{triangle3}
Suppose $\cA$ is a zero-nonzero pattern and $A\in\A$ is a nilpotent matrix that allows a full-rank Jacobian. 
Suppose $u\neq v$ and 
$\det A(u,v) \neq 0$.  
If  $\cB$ is a $\oast$-extension of $\cA$ at $(u,v)$,  
then every superpattern of $\cB$ is spectrally arbitrary. 
\end{theorem}

\begin{proof} The proof is the same as Theorem~\ref{triangle} with some simplification.
In particular, with no loop moved to $w$, the
equation (\ref{charY}) can simplify to
\begin{equation}\label{charY2}
{p}_{Y}(z) = zp_X(z) + 
\sum_{r=3}^{n}S_rx_{m+1}z^{n-r+1} + x_{m+1}\left(\det A(u,v)\right)  
\end{equation}
assuming one of $(u,w)$ and $(w,v)$ is weighted 1 and the
other $x_{m+1}$.
In this case, $B$ can be seen to be nilpotent if $x_{m+1}$ is set to zero.
Finally, $\mathbf{s}^T=[0,0,S_3,\ldots,S_n, \det\left(A(u,v)\right)]$
and 
$\mathbf{t}^T=[0,\ldots,0,1]$, so that the $J_{Y=B}$ will have full rank.
\end{proof}

\begin{remark}{\rm 
As with Theorem~\ref{triangle}, Theorem~\ref{triangle3} can be applied recursively on
one of the two new edges, using a nilpotent realization for which 
the second added edge is nonzero. 
}
\end{remark}

\begin{example}{\rm 
One can show that the 
nested intercyclic digraphs described in~\cite{EKSV} are $\oast$-extensions of the pattern
$\cC_2$, with digraph given in Figure~\ref{C2}, starting with edge $(2,1)$.
It follows from Theorem~\ref{triangle3} and \cite[Theorem 3.2]{EKSV} that 
the Fiedler companion patterns characterized in~\cite{EKSV} 
are recursively constructed spectrally arbitrary $\oast$-extensions of
$\cC_2$.
\begin{figure}[ht]
\begin{center}
\begin{tabular}{cc}
$\cC_2 = 
\left[\begin{array}{cc}
\oast & *\\
\oast&0
\end{array}\right]$ \qquad
&
\qquad
\tikzpicture
\node (1) at (-0.5,-0.5)[place] {};
\node (4) at (1,-0.5)[place]{}; 
\draw [-] (-0.5,-0.5) arc (360:0:3pt);
\draw [->](1) to [bend left=15] (4);
\draw [->](4) to [bend left=15] (1);
\endtikzpicture
\end{tabular}
\end{center}\caption{The pattern $\cC_2$ and its digraph.}\label{C2}
\end{figure}
}
\end{example}

\begin{remark}{\rm 
Note that in the definition of $\oast$-extension, we do not require $(u,v)$ to be an arc of 
the digraph, unlike for triangle extensions. This means that there is much more flexibility 
to construct spectrally arbitrary $\oast$-extensions than triangle extensions.
}\end{remark}

\begin{example}{\rm 
Recall that the pattern $\cT_3$ is spectrally arbitrary and allows a full-rank Jacobian. There is only one
transversal in the matrix obtained by deleting row 3 and column 1 of $\cT_3$. 
Thus, the determinant condition of Theorem~\ref{triangle3} is satisfied and 
the pattern with the digraph in Figure~\ref{new} is a spectrally arbitrary
$\oast$-extension of $\cT_3$ at $(3,1)$.
\begin{figure}[ht]
\begin{center}
\begin{tabular}{cc}
$\left[\begin{array}{cccc}
\oast & \oast&0&0\\
\oast&0 &\oast&0\\
0&\oast&\oast&\oast\\
\oast&0&0&0
\end{array}\right]$ \qquad
&
\qquad
\tikzpicture
\node (1) at (-0.5,-0.5)[place] {};
\node (4) at (1,-0.5)[place] {};
\node (8) at (2.5,-0.5)[place]{};
\node (3) at (1,-1.5)[place]{};
\draw [right] (3) to (1);
\draw [right] (8) to (3);
\draw [-] (-0.5,-0.5) arc (360:0:3pt);
\draw [-] (2.75,-0.5) arc (360:0:3pt);
\draw [->](1) to [bend left=15] (4);
\draw [->](4) to [bend left=15] (1);
\draw [->](4) to [bend left=15] (8);
\draw [->](8) to [bend left=15] (4);
\endtikzpicture
\end{tabular}\end{center}
\caption{An $\oast$-extension of $\cT_3$}\label{new}
\end{figure}

}\end{example}

Finally, as was done in Section~\ref{triangleIAP}, we can apply an adjusted method
to obtain inertially arbitrary $\oast$-extensions of inertially
arbitrary patterns:

\begin{corollary}
Let $\cA$ be a zero-nonzero pattern. Suppose $A\in\A$ is a matrix with refined
inertia $\ri{A}=(0,0,c_1,c_2)$ for some $c_1\geq 2$ and $A$ allows a full-rank Jacobian
(and hence $\A$ is inertially arbitrary). 
Suppose $u\neq v$ and 
$\det A(u,v) \neq 0$.  
If  
$\cB$ is an $\oast$-extension of $\cA$ at $(u,v)$,  
then  
every superpattern of $\cB$ is inertially arbitrary. 
\end{corollary}

\textbf{Acknowledgements} The third author thanks Hannah Bergsma for an initial conversation
about triangle extensions. We thank Dale Olesky and Pauline van den Driessche for comments that
improved Remarks~\ref{det} and \ref{hyp}. Research supported in part by NSERC Discovery Grant 203336.

\end{document}